\documentclass[11pt, reqno, a4paper]{article}

\usepackage{authblk}
\usepackage{amsfonts, amsthm, amsmath, amssymb}
\usepackage{thmtools}
\usepackage[hypertexnames=false,colorlinks=true,linkcolor=green]{hyperref}
\usepackage{cleveref}
\usepackage{url}
\usepackage{xcolor}
\usepackage[normalem]{ulem}
\usepackage{tikz}
\usepackage{pgfplots}
\usepackage{graphicx}
\usepackage[caption=false]{subfig}

 \newtheorem{thm}{Theorem}[section]
 \newtheorem{cor}[thm]{Corollary}
 \newtheorem{lem}[thm]{Lemma}
 
 \newtheorem{defn}[thm]{Definition}
 \newtheorem{rem}[thm]{Remark}
 \numberwithin{equation}{section}

\newcommand{\N}{\ensuremath{\mathbb{N}}}
\newcommand{\R}{\ensuremath{\mathbb{R}}}
\newcommand{\C}{\ensuremath{\mathbb{C}}}

\newcommand{\Z}{\ensuremath{\mathbb{Z}}}

\newcommand{\hX}{\ensuremath{\hat{X}}}
\newcommand{\hY}{\ensuremath{\hat{Y}}}
\newcommand{\hZ}{\ensuremath{\hat{Z}}}

\newcommand{\hU}{\ensuremath{\hat{U}}}

\title{Pontryagin Maximum Principle in Free Probability Theory}

\providecommand{\keywords}[1]
{
	\small	
	\noindent\textbf{\textbf{Keywords }} #1
}

\providecommand{\amscodes}[1]
{
	\small	
	\noindent\textbf{\textbf{AMS Codes }} #1
}

\date{\today}

\author[1, 2]{Georg Schluechtermann}
\author[2]{Michael Wibmer}

\affil[1]{\footnotesize Faculty of Mathematics, Informatics and Statistics, LMU Munich, Germany}
\affil[2]{\footnotesize Faculty of Mechanical, Aeronautical and Automotive Engineering\\ University of Applied Sciences, Munich, Germany}

\begin{document}

\noindent

\maketitle

\begin{abstract}
	Motivated by the classical stochastic maximum principle, random matrices and free stochastic differential equations we, develop an analog maximum principle for control problems driven by non-commutative random variables, e.g. random matrices. We formulate an optimal control problem in the setting of free probability, consisting of the controlled forward equation, a free backward stochastic differential equation. For both, we give global existence theorems. Due to the non-commutative It\^{o}-formula, the definition of the Hamiltonian differs to the commutative case. Our strategy is to stay as close as possible to the commutative case. Finally we formulate and proof the maximum principle in the context of free probability. Several examples show the application of the maximum principle, where explicit solutions can be found. 
\end{abstract}

~\\~\\
\amscodes{47C15, 49K27, 46L54}

\noindent

\section{Introduction}
Our starting point is the classical stochastic optimal control problem. Given a controlled, $\R^n$ valued stochastic differential equation
\begin{equation}
	dx_t=a(x_t, u_t)dt + b(x_t,u_t)dw_t, 0\leq t \leq T
\end{equation}
where $(w_t)_{t\geq 0}$ is a $d$-dimensional Brownian motion on a filtered probability space $(\Omega, \mathcal{F}, \mathbb{F}=(\mathcal{F})_{t\geq 0}, P)$. $x_0\in\R^d$ is the initial condition at $t=0$. $a,b$ are functions with certain regularity properties to ensure existence and uniqueness of solutions of the stochastic differential equation (see \cite{Pham2009}). The process $u=(u_t)$ is  $\R^n$ valued, progressively  measurable (with respect to $\mathbb{F}$). In this paper, we focus on finite horizon control problems. Let $0<T<\infty$. Then the aim is to find a control $\hat{u}$,  such that the gain function
\begin{equation*}
	J(u)=\mathbb{E}\left[ \int_0^T f(s,x_s , u_s)ds + g(x_T )\right]
\end{equation*}
is maximized over some suitable set of controls, i.e. $\sup_{u\in \mathcal{U}}J(u)$.
We refer to \cite{Pham2009}, \cite{Yong1999} for more details. There are two classical ways to solve the optimal control problem. The first approach is the so called dynamical programming principle (DPP), developed in the 1950s by Bellman. In general, this attempt leads to a second order and nonlinear partial differential equation, the so called Hamilton-Jacobi-Bellman equation (HJB). Solutions to the HJB equation are candidates for the optimal solution of the control problem. For details about DPP we refer to \cite[Section 3]{Pham2009}. The second approach is the so called stochastic maximum principle (SMP), which is the generalization of the classical Pontryagin optimality principle to dynamical systems, described by stochastic differential equations (SDEs). Necessary conditions for optimality rely on a backward stochastic differential equation (BSDE) and a Hamiltonian function $h$. Pontryagin's principle is expressed in an optimality condition on $h$ based on the solution of the BSDE. For details about SMP we refer to \cite[Section 6]{Pham2009}. In \Cref{sec:preliminiaries} we formulate the classical SMP in more detail.\\

The aim of this paper is to carry over the general stochastic maximum principle to the setting of free probability resp. non-commutative random variables. At a first glimpse, one can consider symmetric $n\times n$ random matrices as random variables. In general, we consider a non-commutative probability space as a pair $(\mathcal{A}, \varphi)$, where $\mathcal{A}$ is a von Neumann algebra with a faithful, unital, normal trace $\varphi$ (\cite{MingoSpeicher2017}).\\

Stochastic optimization problems in the non-commutative setting appeared in the late 90's by studying large deviations for matrix Brownian motions (\cite{Biane2003}). Hamilton-Jacobi-Bellman equations in the non-commutative setting appeared first in \cite{Jekel2020}. 
We refer to \cite{gangbo2025viscsol} for a brief summary about works concerning applications of optimal control e.g. in free optimal transport and free stochastic analysis.\\
A maximum principle for optimal control of stochastic evolution equations was considered in \cite{dukaimeng}. Stochastic optimal control of interacting particle systems in Hilbert spaces were treated in \cite{deFeooptimcontinHilbert}. We refer to \cite{Boscain2021} for an introduction to the Pontryagin maximum principle for Quantum Optimal Control. Control in Hilbert space and first order mean field type problems can be found in \cite{bensoussan2021controlhilbertspaceorder}.\\
 
 Recently Ganbo et. al introduced in \cite{gangbo2025viscsol} stochastic optimal control problems in the free probability setting and considered their solution by help of the dynamic programming principle. The forward equation was a free differential equation with a constant diffusion factor. Additionally, they added a single classical Brownian motion to the differential equation, motivated by parallels between mean field games and random matrix theory. Dabrowksi applied optimal control in \cite{dabrowski2017laplaceprinciplehermitianbrownian} in the context of free entropy. The underlying free stochastic differential equation was driven by a constant diffusion factor, too. In \cite{das2025freeprobabilisticframeworkdenoising}, the author Das extended diffusion modeling to the operator valued framework. The forward dynamics was realized by a free Ornstein-Uhlenbeck flow, the reverse-time dynamics appeared as free deconvolution process driven by the conjugate variable. \\

To the best knowledge of the authors, a stochastic maximum principle in the context of general free stochastic differential equations has not been formulated before. In \cite{WangWangFermion}, the authors considered a stochastic maximum principle in the context of fermion Brownian motion. While general quantum probability theory relies on tensor products and defines independence via commuting subsystems, free probability theory utilizes free products of algebras for strictly non-commuting variables. The fundamental distinction lies in this notion of freeness, which leads to distinct limit theorems such as the Wigner semicircle law rather than the classical Gaussian distribution.\\

To formulate control problems, we extend free stochastic differential equations (\cite{Kargin}, \cite{SchlueWib2023}, \cite{wibmer_schlue_milstein2026}, \cite{kummererspeicher}, \cite{BIANESPEICHER2001581}) by a control variable and prove global solutions assuming global Lipschitz property of the underlying drift function as well as to the product of the diffusion functions. 
 We will formulate the stochastic maximum principle based on a forward backward system of stochastic differential equations. We prove the maximum principle allowing for nonlinear drift and diffusion terms of the underlying free forward equation. Surprisingly the formulation of the free version of the maximum principle is closer than expected to the commutative case. The maximum principle in the non-commutative setting naturally requests for differentiating operator functions, which map (self-adjoint) operators from $\mathcal{A}$ to $\mathcal{A}$. Such derivatives can be expressed by help of multiple operator integrals (MOIs, \cite{Skripka2019}). They are a powerful tool, but rather inflexible in terms of dealing with applications. Instead of dealing directly with MOIs, we express the Hamiltonian and their derivatives in $\R$ and applying the functional calculus  to transfer these functions to the non-commutative setting. The stochastic maximum principle and the backward stochastic differential equations are then formulated via the Hamiltonian function $H$. There are two main differences to the commutative case. First, the non-commutative version of the It\^{o}-formula (\Cref{sec:freeItoFormula}, \cite[Theorem 5.3.13]{Biane1998}) imposes additional terms in $H$, which do not appear in the commutative case. Second, the martingale representation theorem (\Cref{sec:freeCalc-BiprocStochInt}) is needed in order to formulate a backward equation. Unfortunately, the representation theorem request the use of general biprocesses, defined in \cite{Biane1998}. It turns out that, although the stochastic maximum principle as a theorem is rather similar to the commutative case, the applicability turns out to more challenging. The It\^{o}-formula and biprocesses generate a Hamiltonian, which is difficult to work with. Nevertheless, we give linear examples, for which an explicit solution can be found. The solution process is aligned to the commutative case and the examples show the differences to the non-commutative setting.\\
 
 The paper is organized as follows. In \Cref{sec:preliminiaries} we formulate preliminaries, necessary tools and fix the notation. In \Cref{sec:fSDE} formulate a global existence theorem on finite intervals. In \Cref{sec:defoptimcontrolproblem} we formulate the optimal control problem and \Cref{sec:BSDEsandFreeOptimControl} contains the definition of the Hamiltonian, the BSDE and the final SMP. In \Cref{sec:examples} we first present the aforementioned linear examples and present the derivation of a Riccati differential equation. 
 
\section{Preliminaries}\label{sec:preliminiaries}
In this section we summarize the necessary and well known ingredients for the following chapters. 
\subsection{Classical Stochastic Maximum Principle (SMP)}\label{sec:classicalsmp}
The classical stochastic Pontryagin principle acts as a template for the formulation of the Pontryagin principle in the free probability setting. Therefore we repeat the classical principle, as it is referenced in the following chapters. For the background and details on all the necessary assumptions and conditions, we refer to \cite{Pham2009} and \cite{Yong1999}. We will mainly follow \cite{Pham2009} in the following.\\
Consider a filtered probability space $(\Omega, \mathcal{F}, \mathbb{F}=(\mathcal{F}_t)_{t\geq 0}, \mathbb{P})$. $X_t(\omega)$ describes the $\R^n$-valued state of a system at time $0\leq t\leq T<\infty$ in a world scenario $\Omega$. The time dynamics of the state is described by the stochastic differential equation
\begin{equation}
	dx_t=a(t,x_t, u_t)dt + b(t,x_t,u_t)dw_t.
\end{equation}
$w=(w_t)_{t\geq 0}$ is a $n$-dimensional Brownian motion on the underlying filtered probability space. The drift and diffusion function fulfill certain conditions to ensure global existence on the finite interval $[0,T]$. The control $u=(u_t)_{t\geq 0}$ is a progressively measurable process (with respect to $\mathcal{F}$) in a suitable subset of $U\subset \R^m$. For details on the assumptions we refer to \cite{Pham2009}.\\

For the stochastic maximum principle we take a function $f:[0,T]\times \R^n\times U\rightarrow \R$ continuous in $(t,x)$ for all $u\in U$  and a concave function  $g:\R^n \rightarrow \R$. Assuming both $f$ and $g$ satisfy a quadratic growth condition in $x$, we can define a so called gain function $J$ as 
\begin{equation}
	J(u)=\mathbb{E}\left[\int_0^Tf(x,x_s, u_s)ds + g(x_T)\right].
\end{equation}
for $u \in U$. Due to the quadratic growth condition on $f$, the set of controls $\mathcal{U}$ does not depend on the initial condition (\cite[Remark 3.2.1]{Pham2009}. The gain function $J$ is finite over all controls. The objective is to maximize $J$ over control processes
\begin{equation}
   \sup_{u\in U}J(t,x,u).
\end{equation}

We define the generalized Hamiltonian $H:[0, T] \times \mathbb{R}^n \times U \times \mathbb{R}^n \times \mathbb{R}^{n \times d} \rightarrow \mathbb{R}$ by

\begin{equation}\label{eq:genhamilton-classic}
	H(t, x, u, y, z)=b(x, u) . y+\operatorname{tr}\left(\sigma^{T}(x, u) z\right)+f(t, x, u), 
\end{equation}

and assume that $H$ is differentiable in $x$ with derivative denoted by $D_x H$. We consider for each $u \in U $, the BSDE, called the adjoint equation:

\begin{equation}\label{eq:BSDE-classic}
	-d y_t=H_x\left(t, x_t, \alpha_t, y_t, z_t\right) d t-z_t d w_t, \quad y_T=D_x g\left(x_T\right) 
\end{equation}

The stochastic maximum principle in the classical theorem, allowing for nonlinear drift and diffusion terms, is given in  \cite[Theorem 6.4.6]{Pham2009}. It is the template for the main \Cref{thm:freestochprinciple-1}.
\begin{thm}\label{thm:pontryagin-classic}
	
Let $\hat{u} \in U$ and $\hat{x}$ the associated controlled diffusion. Suppose that there exists a solution $(\hat{y}, \hat{z})$ to the associated BSDE \eqref{eq:BSDE-classic}, such that

\begin{equation*}
	H\left(t, \hat{x}_t, \hat{u}_t, \hat{y}_t, \hat{z}_t\right)=\max _{u \in U} H\left(t, \hat{x}_t, u, \hat{y}_t, \hat{z}_t\right), \quad 0 \leq t \leq T, \quad \text { a.s. } 
\end{equation*}

and

\begin{equation*}
	(x, u) \rightarrow H\left(t, x, u, \hat{y}_t, \hat{z}_t\right) \quad \text { is a concave function, }
\end{equation*}

for all $t \in[0, T]$. Then $\hat{u}$ is an optimal control, i.e.

$$
J(\hat{u})=\sup _{u \in U} J(u) .
$$

\end{thm}

As an example, consider the linear quadratic optimal control (LQR) problem. We briefly summarize, how to obtain the well known matrix Riccati differential equation out of the maximum principle for a linear quadratic optimal control problem (see e.g. \cite{Yong1999}).\\
The system of an LQR-problem is given by the stochastic differential equation
\begin{equation}\label{eq:LQR_SDE_classical}
	dx_t=(Ax_t+Bu_t)dt + Du_tdw_t
\end{equation}
on $[0,T]$. Let $u_t\in \R^m$, $x_t\in\R^n$. Furthermore $A,C\in \R^{n\times n}$, $B\in\R^{n\times m}$ are constant matrices over $\R$. $(w_t)_{t\geq 0}$ is a classical Brownian motion. The cost functional is given by
\begin{equation}
	J(u)=-1/2\mathbb{E}\left[\int_0^T x_s^TQx_s + u_s^TRu_s ds + x_T^TMx_T\right],
\end{equation}
where $M, Q\in\R^{n\times n}$ are positive semidefinit, symmetric and $R\in\R^{m\times m}$ is positive definite. We define $H:[0,T]\times\R^n\times\mathbb{U}\times \R^n \times \R^{n \times d}\rightarrow \R$, the so called 
generalized Hamiltonian, as
\begin{equation}\label{eq:classical_generalized_Hamilton}
	H(t,x,u,y,z)=(Ax+Bu)\cdot y+ Du\cdot z - 1/2\left(x^TQx + u^TRu\right).
\end{equation}
Differentiating $H$ w.r.t. $x$ shows, that
\begin{equation}
	H_x(t,x,u,y,u)=A^Ty-Qx.
\end{equation}
The so called backward stochastic differential equations (BSDE) is given by
\begin{equation}\label{eq:BSDE_LQR_classical}
	-dy_t=(A^Ty_t+Qx_t)dt + z_tdw_t, ~ y_T=-Mx_T.
\end{equation}
The optimality condition on the generalized Hamiltonian $H$ in \Cref{thm:pontryagin-classic} gives
\begin{equation}\label{eq:h3_deriv_matrixriccati}
	H_u(t,x,u,y,z)=B^Ty+D^Tz-Ru.
\end{equation}
Setting $H_u=0$ we obtain 
\begin{equation}\label{eq:control}
	u=R^{-1}(B^Ty+D^Tz).
\end{equation}
To solve the LQR-problem we make the ansatz $y=Px$ with some deterministic $P:[0,T]\rightarrow \R^{n\times n}$.
Inserting the ansatz into \eqref{eq:BSDE_LQR_classical} gives
\begin{equation}\label{eq:h1_deriv_matrixriccati}
	-dPx_t-Pdx_t=(A^Ty_t+Qx_t)dt - z_tdw_t.
\end{equation}
Inserting \eqref{eq:LQR_SDE_classical} into \eqref{eq:h1_deriv_matrixriccati} yields
\begin{equation}\label{eq:h2_deriv_matrixriccati}
	-dPx_t-P(Ax_t+Bu_t)dt - PDu_tdw_t = (A^Ty_t+Qx_t)dt - z_tdw_t.
\end{equation}
Due to \eqref{eq:control} and comparing the diffusion term on the left and right in \eqref{eq:h2_deriv_matrixriccati} shows $z_t=PDu_t$. Then a comparison of the drift coefficients in \eqref{eq:h2_deriv_matrixriccati} results in the classical matrix-valued Riccati differential equation
\begin{equation}\label{eq:matrixriccati-classical}
	dP+A^TP+PA+PB(R-D^TPD)^{-1}B^TP-Q=0.
\end{equation}
There are more general cases, for example if the diffusion is also linear in $u$. For details about the solvability and more general cases we refer to \cite{Yong1999}. In \Cref{seq:LQR-example} we consider the linear quadratic optimal control problems in the context of free probability. The aim is to mimic the solution process as presented above. 

\subsection{Free Probability}
Free probability was created by C. Voiculescu in mid 1980's by studying properties of von Neumann algebras. He introduced the notion of freeness, which extends the notion of independent random variables to non-commutative setting. He also discovered, that random matrices satisfy the freeness conditions asymptotically. By the help of non-commutative algebras it is possible to develop non-commutative probability theory. The limits $N\rightarrow\infty$ can be handled properly in algebraic structures and lead to fruitful concepts. It turns out that non-commutative probability theory is realized by using operator algebras such as von Neumann algebras. We refer e.g. to   \cite{anderson_guionnet_zeitouni_2009}, \cite{MingoSpeicher2017}, \cite{TaoIntroRMT} and references therein, for setting up non-commutative probability theory and relations to random matrices. To be complete, we give the following general definition (see e.g. \cite{MingoSpeicher2017}). 
\begin{defn}
	A non-commutative probability space is a pair $(\mathcal{A}, \varphi)$, where $\mathcal{A}$ denotes a von Neumann operator algebra and $\varphi:\mathcal{A}\rightarrow \C$ a faithful unital normal trace.
\end{defn}
Since we consider von Neumann algebras with a unital, faithful and normal trace $\varphi:\mathcal{A}\rightarrow\C$, we can introduce for $1\leq p<\infty$ a norm on $\mathcal{A}$ by $\|X\|_p=\varphi(|X|^p)^{\frac{1}{p}}$. The Banach space completion of $\mathcal{A}$ by $\|\cdot\|_p$ is denoted by $L_p(\varphi)$ (see e.g. \cite{PISIER20031459}).
Since the trace is finite we may consider $\mathcal{A}$ as a subset of the predual $L_1(\varphi)$ of the von Neumann algebra $\mathcal{A}=L_\infty(\varphi)$. 
By $\| \cdot \|$ we denote the usual operator norm in $\mathcal{A}$. Let $\mathcal{A}^{sa}=\{a\in\mathcal{A}, a^*=a\}$. For a non-commutative random variable $X\in\mathcal{A}^{sa}$,	there is a unique probability measure on $\R$ with compact support having the same moments as $X$ (e.g. \cite{MingoSpeicher2017}, \cite{TaoIntroRMT}). This probability measure is the distribution of the non-commutative random variable X. We state the following definition, which is from \cite{anderson_guionnet_zeitouni_2009}:
\begin{defn}
	Let $\mathbb{X}=(X_i)_{i\in J}$ be a family of elements of $(\mathcal{A}, \varphi)$. $\mathbb{C}\langle\mathbb{X}\rangle$ denotes the set of polynomials in $X_i$. The distribution (or law) of $\mathbb{X}$ is the map $\mu_\mathbb{X}:\mathbb{C}\langle\mathbb{X}\rangle\mapsto \C$ such that
	$$
	\mu_\mathbb{X}(P)=\varphi(P(\mathbb{X})).
	$$
\end{defn}
The non-commutative analog of independence of classical random variables is the concept of free independence, or shortly freeness, of subalgebras of $\mathcal{A}$ (\cite{MingoSpeicher2017}). Let $\mathcal{A}_1,\dots \mathcal{A}_n$ be a family of $n\in\N$ subalgebras of $\mathcal{A}$. They are called freely independent (or simply free) in the sense of Voiculescu, if $\varphi\left(X_1X_2\dots X_m\right)=0$ 
whenever the following conditions
\begin{enumerate}
	\item $X_j\in\mathcal{A}_{i(j)} $, where $i(1)\neq i(2), i(2)\neq i(3), \dots , i(n-1)\neq i(n)$, $j=1,\dots,m$
	\item $\varphi(X_i)=0$ for all $i=1,\dots,n$
\end{enumerate} 
hold (\cite[Definition 11]{MingoSpeicher2017}). If $X\in\mathcal{A}$ is a self-adjoint element, then there is a unique spectral measure $\mu$ on $\R$ so that the moments of $X$ are the same as the moments of the probability measure $\mu$ defined by
$
\varphi(X^k)=\int_\R x^k d\mu(x), 
$
see \cite[pp. 51]{MingoSpeicher2017}. An important role plays the Cauchy transform $G_X$ of $\mu$ defined by
$
G_X(z)=\int_{\R} \frac{d\mu(x)}{x-z},
$
which is an analytic function defined on $\C^+$ with values in $\C^+$. The Cauchy transform $G_X$ is the expectation of the resolvent of $X$, i.e.
$
G_X(z)=\varphi\left(\left(X-z\right)^{-1}\right)
$.\\
The Cauchy transform carries all the properties of the spectral probability distribution of the self-adjoint operator $X$. 
In \cite{BIANESPEICHER2001581} the authors used the Hilbert transform of the solution of the underlying fSDE to obtain detailed information about the distribution of the solution. In \cite{Kargin} these results were extended to general fSDEs. They can be handled by a corresponding deterministic partial differential equations of the Cauchy transform $G_X$.

\subsection{Free Stochastic Calculus}
The analog of classical stochastic calculus is called free stochastic calculus and was developed in the late 90's by P.~Biane and R.~Speicher in \cite{Biane1998}. The notion of Brownian motion in the classical sense can be carried over to the non-commutative setting by replacing  \emph{independent increments} by \emph{free increments} and the normal distribution by the semicircle. It is then possible to setup the concept of a free stochastic integral, which contains left- and right integrands due to the non-commutativity. The free stochastic integral allows for a Burkholer-Gundy inequality even valid in $L_\infty(\varphi)$. The concept of stochastic processes is extended to so called biprocesses, which are the completion of so called simple (time-constant) processes in von Neumann algebras. The It\^{o}-formula is a key ingredient in this paper. 
In the following we briefly introduce theses topics. An good overview can be found in \cite{speicher2001freecalculus}.\\

 For the rest of the paper, consider a von Neumann algebra $\mathcal{A}$ with a faithful normal trace  $\varphi:\mathcal{A}\rightarrow \C$.
A filtration $\mathbb{F}=(\mathcal{A}_t)_{t \geq 0}$ is a family of von Neumann subalgebras $\mathcal{A}_t$ of $\mathcal{A}$ with $\mathcal{A}_s \subset \mathcal{A}_t$ for $s\leq t$.
A family of elements $(X_t)_{t\geq 0}\subseteq \mathcal{A}$ and is called adapted to the filtration $\mathbb{F}$ if $X_t\in\mathcal{A}_t$ for all $t\geq 0$. 

\subsubsection{Free Brownian Motion}
\label{sec:FreeBrownMotion}
Motivated by the concept of classical Brownian motion the definition within non-commutative probability is as follows. Consider a von Neumann algebra $\mathcal{A}$ with a faithful normal trace  $\varphi:\mathcal{A}\rightarrow \C$.
A filtration $\mathbb{F}=(\mathcal{A}_t)_{t \geq 0}$ is a family of von Neumann subalgebras $\mathcal{A}_t$ of $\mathcal{A}$ with $\mathcal{A}_s \subset \mathcal{A}_t$ for $s\leq t$.
A family of elements $(X_t)_{t\geq 0}\subseteq \mathcal{A}$ is called adapted to the filtration $\mathbb{F}$, if $X_t\in\mathcal{A}_t$ for all $t\geq 0$. 
\begin{defn}
	A free Brownian motion $(W_t)_{t\geq 0}$ is a family elements of $\mathcal{A}$ adapted to the filtration $\mathbb{F}=(\mathcal{A}_t)_{t\geq 0}$, which admits the following properties:
	\begin{enumerate}
		\item $W_t$ is a self-adjoint element of $\mathcal{A}$ with semi-circular distribution of mean zero and variance $t$, 		
		\item For all $s,t$ with $s\leq t$, the element $W_t-W_s$ is free of $\mathcal{A}_s$ and has a semi-circular distribution with mean $0$ and variance $t-s$.
	\end{enumerate}
\end{defn}
Free Brownian motion as a random variable in a von Neumann algebra  $\mathcal{A}$ is uniformly bounded in operator norm on finite time intervals, i.e. $\sup_{t\in[0,T]}\|W_t\|<\infty, \, 0<T<\infty$.

\subsubsection{Biprocesses and Free Stochastic Integral}\label{sec:freeCalc-BiprocStochInt}
In \Cref{sec:FreeMaxPrinc:BSDE}, we are going to formulate a free analog of the classical backward stochastic differential equation. Due to the martingale representation theorem (\cite[Proposition 5.3.13]{Biane1998}), the diffusion term consists of an adpated biprocess $Z\in\mathcal{B}_2^a$. The space $\mathcal{B}_2^a$ was developed in \cite{Biane1998} in order to define stochastic integration with respect to free Brownian motion, avoiding the construction via left- and right-processes on the free Fock space on $L_2(\R_+)$. $\mathcal{B}_2^a$ is the completion of the vector space of simple biprocesses with respect to the norm
$$
\|U\|_{\mathcal{B}_2([0, T])}=\left(\int_0^T\left\|U_s\right\|_{L^2(\varphi \otimes \varphi^{op})}^2 ~ds\right)^{1 / 2}.
$$
It is a Hilbert space associated with the inner product
$$
\langle U,V\rangle = \int_0^\infty \langle U_s,V_s\rangle~ds,
$$
where $\langle U_s,V_s\rangle$ is the inner product in $L^2(\mathcal{A},\varphi)\otimes L^2(\mathcal{A},\varphi)$. Adapted biprocesses form a closed subset of $\mathcal{B}_2$ and will be denoted by $\mathcal{B}_2^a$. \\

To handle biprocesses, additionally we need the opposite algebra $\mathcal{A}^{op}$ to $\mathcal{A}$, with the trace $\varphi^{op}$. The superscirpt highlights the algebraic structure, but as a linear map  $\varphi=\varphi^{{op}}$. The spaces $\mathcal{A}$ and $\mathcal{A} \otimes \mathcal{A}$ have a left $\mathcal{A} \otimes \mathcal{A}^{\mathrm{op}}$ module structure and defined the corresponding actions as $(a \otimes b) \sharp u=a u b$ and $(a \otimes b) \sharp(u \otimes v)=a u \otimes v b$. Of course, the action of $\mathcal{A} \otimes \mathcal{A}^{{op}}$ on $\mathcal{A} \otimes \mathcal{A}$ corresponds to the multiplication on the left in the algebra $\mathcal{A} \otimes \mathcal{A}^{{op}}$. The map $\varphi \otimes \varphi^{{op}}$ defines a tracial state on the $*$-algebra $\mathcal{A} \otimes \mathcal{A}^{{op}}$, and we shall denote by $L^p\left(\varphi \otimes \varphi^{{op}}\right)$ the corresponding $L^p$-spaces, thus $L^{\infty}\left(\varphi \otimes \varphi^{{op}}\right)$ is the von Neumann algebra tensor product of $\mathcal{A}$ and $\mathcal{A}^{{op}}$.\\

In the classical SMP, the generalized Hamiltonian \eqref{eq:classical_generalized_Hamilton} contains the derivative of the Hamiltonian w.r.t. to the space variable. This requires more knowledge about the representation of biprocesses.\\
According to \cite[Chapter 5]{Biane1998}, an adapted biprocess $Z\in\mathcal{B}_2^a$ can be decomposed into an infinite sum of a tensor product of so called multiple integrals. In the classical case, the space $L_2(\mu)$ (with Wiener measure $\mu$) can be decomposed into multiple integrals by the so called It\^{o}-Wiener chaos decomposition. Similarly, in the free context, elements from the free (unsymmetric) Fock space over $L_2(\R_+)$ can be decomposed as an infinite sum of free multiple integrals. In \cite[Chapter 5]{Biane1998} it was shown, that this is also valid for any adapted biprocess $\mathcal{B}_2^a$. For $Z\in\mathcal{B}_2^a$ we write $U=\sum_{j\in\Z}A_j\otimes B_j$ with adapted $A_j, B_j\in\mathcal{A}$.

\subsection{Free It\^{o}-Formula}\label{sec:freeItoFormula}
A free analog of the classical It\^{o}-formula was developed in \cite{Biane1998}. Given a free Brownian motion $(W_t)_{t\geq 0}$, then the It\^{o}-formula in product form reads
$$
 X dW_t X'\,Y dW_tY' = XY'\varphi(X'Y)dt,
$$
for $X,X',Y,Y'\in\mathcal{A}_t$. In \Cref{sec:generalizedhamilton} when defining the backward equation, we will need the It\^{o}-formula in terms of biprocesses. The product is then formulated by help of the algebraic tensor product in $\mathcal{A}\otimes \mathcal{A}$, see \cite[Definition 4.1.1]{Biane1998}. Consider a biprocess of the form $U_t=\sum_{i\in\Z}A_{i,t}\otimes B_{i,t}$. Then in principle, the It\^{o}-formula in product form reads
$$
 U_t\#dW_t Y dW_t Y'=  \sum_{i\in\Z} A_{i,t}dW_t B_{i,t}YdW_t Y' =  \sum_{i\in\Z} A_{i,t}Y'\varphi(B_{i,t}Y)dt.
$$
The adjoint variable $Z$ in the backward stochastic differential equation is a  biprocess, where we have to assume, that $Z$ must be expressed as an infinite sum. This complicates the definition of the Hamiltonian. We refer to \Cref{sec:generalizedhamilton}.

\subsection{Free Stochastic Differential Equations (fSDEs)}\label{sec:fSDE}
Free stochastic differential equations (fSDE) appeared first in \cite{kummererspeicher}.  In \cite{BIANESPEICHER2001581} the authors considered multidimensional equations, where the equations are driven by a multidimensional free Brownian motion. A set of interesting examples is given in \cite{Kargin}. Numerical Algorithms can be found in \cite{SchlueWib2023}, \cite{wibmer_schlue_milstein2026} and \cite{NiuWeiYinZhen}. Application to optimal transport is discussed in \cite{Jekel2020}, \cite{jekel2021tracial} and \cite{dabrowski2016free}. Evolution Equations are topic in \cite{JekelEvolution} and 
applications to optimal control was presented in \cite{das2025freeprobabilisticframeworkdenoising}.\\

We first briefly define free stochastic differential equations (fSDEs) and the notion of a solution. In the sequel we formulate a global existence theorem on the finite intervals $[0,T]$, by extending the local existence result by \cite{Kargin}. Just as in the classical case for ODEs, we assume global Lipschitz drift and diffusion coefficients to ensure global existence. The existence results in  \cite{BIANESPEICHER2001581} and \cite{Gangbo2022} considered constant diffusion terms, whereas in our case, drift and diffusion are allowed to be  potentially nonlinear. It is possible to prove an existence result weakening the Lipschitz continuity to gain existence with the technique of the measure of non-compactness. 

\begin{defn}\label{def-freeItoProcess}
	Let $X_0$ be a self-adjoint element  in $\mathcal{A}^{sa}$ and $a,b^i,c^i:\mathcal{A}\rightarrow \mathcal{A}$ continuous functions in the operator norm (resp. in the $L_p(\varphi)$-norm). We call 
	\begin{equation}\label{intro-freeSDE-diffform-orig}
		dX_t=a(X_t)dt+ \sum\limits_{i=1}^d b^i(X_t)dW_tc^i(X_t)
	\end{equation}
	a (formal) free Stochastic Differential Equation (fSDE).	
	A solution of \eqref{intro-freeSDE-diffform-orig} on $[0,T]$ with initial condition $X(0)=X_0$ is a process $(X_t)_{t\geq 0}$ with the following properties:
	\begin{enumerate}
		\item $X(0)=X_0$ is a self-adjoint element in $\mathcal{A}^{sa}$
		\item $X_t\in\mathcal{A}_t^{sa}$ for all $t\in [0,T]$
		\item The equation
		\begin{equation}\label{intro-freeSDE-inform-orig}
			X_t=X_0 + \int_0^t a(X_s)ds + \sum\limits_{i=1}^d\int_0^tb^i(X_s)dW_sc^i(X_s)
		\end{equation} 
		is fulfilled for all $t\in[0,T]$.
	\end{enumerate}
\end{defn}
As an example, consider the equation $dX_t=\theta X_tdW_tX_t$ with $X_0=\alpha I$. Following \cite{Kargin}, the spectral distribution of $X_t$ is supported on the interval 
	$$[\frac{(1-\alpha\theta\sqrt{t})^2}{(1-\alpha^2\theta^2t)^2}, \frac{(1+\alpha\theta\sqrt{t})^2}{(1-\alpha^2\theta^2 t)^2}],
	$$ if $t<(\alpha \theta)^{-2}$. 
	The solution explodes in finite time, i.e.
	$$
	\lim\limits_{t\rightarrow (\alpha \theta)^{-2}}\|X_t\|=\infty.
	$$
	On every interval $[t_0, \tau]$, where $\tau<(\alpha \theta)^{-2}$, we have $\|X_t\|<M(X_0)$ for $t_0\leq \tau$, where the bound depends on the initial value $X_0$.

\begin{defn} A function 
	$f:\mathcal{A} \rightarrow \mathcal{A}$ is called  locally operator Lipschitz in $L_p(\varphi),~1\leq p \leq \infty$, if there is a constant $L_f(K), K>0$ such that
	$$
	\|f(X_1)-f(X_2)\|_p\leq L_f(K)\|X_1-X_2\|_p,
	$$
	for all $\|X_i\|_p \leq K, i=1,2$.
	$f$ is called globally operator Lipschitz in $L_p(\varphi)$, if $L_f$ does not depend on $K$.
\end{defn}
It is worth mentioned, that real functions, which belong to the Besov space $B_{\infty,1}^1(\R)$, are globally operator Lipschitz, see \cite{AleksandrovPeller}.\\

First, we formulate a global existence theorem for equations of the form
\begin{equation}\label{eq:h34}
	dX_t=a(X_t)+b(X_t)dW_tb(X_t),\, 0\leq t \leq T < \infty.
\end{equation}
which only contain a single diffusion term. To meet the self-adjoint requirement, the left- and right factor must be equal.

\begin{lem}\label{lem:fSDE_global_uniform_bound}
	 Let $0<T<\infty$ and $a:\mathcal{A}\rightarrow \mathcal{A}$ be globally operator Lipschitz, i.e. 
	\begin{align}
		\|a(X)-a(Y)\| &\leq L_a \|X-Y\| \label{eq:thmglob-assum1}
	\end{align}
	such that $a(\mathcal{A}^{sa})\subseteq \mathcal{A}^{sa}$.
	Furthermore assume a continuous function $b:\mathcal{A}\rightarrow \mathcal{A}$, such that $b^2$ is at most linear growth, i.e. 
	\begin{align}
		\|b^2(X)-b^2(Y)\| &\leq L_{b}\|X-Y\|\label{eq:thmglob-assum4},
	\end{align}
	where $L_{b}>0$. Suppose \eqref{eq:h34} obeys a solution $(X_t)_{t\in[0,T]}$. Then there is a constant $M>0$, such that $\|X_t\|<M$ uniformly on the finite interval $[0,T]$.
\end{lem}
\begin{proof}
	Let $t\in[0,T]$. Then the Burkholder-Gundy inequality (\cite{Biane1998}) yields
	\begin{align*}
		\|X_t\| &\leq \|X_0\|+\int_0^t\|a(X_s)\|ds + \left\|\int_0^tb(X_s)dW_sb(X_s)\right\|\\
		&\leq \underbrace{\|X_0\| + \int_0^t (\alpha + \beta \|X_s\|)ds}_{C_0} +  \underbrace{2\sqrt{2}\left(\int_0^t\|b(X_s)^2\|\|b(X_s)^2\|ds\right)^{1/2}}_{C_1}.
	\end{align*}
	Taking the square we get 
	$$
	\|X_t\|^2 \leq (C_0+C_1)^2\leq 4\max(C_0^2,C_1^2).
	$$
	Assumptions \eqref{eq:thmglob-assum1} and \eqref{eq:thmglob-assum4} state, that $a$ resp. $b^2$ grow at most linearly. Then Jensen's inequality, and $(a+b)^2\leq 2a^2+2b^2$ show, that there are constants $\alpha_g,\beta_g>0$ such that 
	$$
	C_0^2 \leq 4\max\left(\|X_0\|^2, \int_0^t(\alpha+\beta\|X_s\|)^2ds\right) \leq \int_0^t(\alpha_g+\beta_g\|X_s\|^2)ds
	$$
	and
	$$
	C_1^2\leq 8 \int_0^t(\alpha+\beta\|X\|)^2ds \leq \int_0^t(\alpha_g+\beta_g\|X_s\|^2)ds.
	$$
	Then Gronwall's inequality gives the uniform bound
	$$
	\|X_t\|^2 \leq M=\alpha_gTe^{\beta_g T}, t\in[0,T].
	$$
\end{proof}
Now we can state the global existence theorem.
\begin{thm}\label{thm:fSDE_globalexistence}
	Consider a fSDE as defined in \cref{eq:h34} with similar assumptions as in \Cref{lem:fSDE_global_uniform_bound}.
	Then for every $X_0 \in \mathcal{A}$, there is a unique solution $(X_t)_{0\leq t\leq T}$ of \cref{eq:h34} uniformly bounded on $[0,T]$.
\end{thm}
\begin{proof}
	Theorem \cite[Theorem 3.1]{Kargin} states the existence of a local solution. For every $A>\|X_0\|$ there is a $t_0>0$, such that $X_t$ is a solution on $[0,t_0[$ and $\|X_t\|\leq A$. Due to \Cref{lem:fSDE_global_uniform_bound} we see that $\lim_{ t\rightarrow t_0} \|X_t\|\leq A\leq M<\infty$ and a new initial value $X_{t_0}\in\mathcal{A}^{sa}$ can be defined at time point $t_0$. Then apply the local existence result again. Due to the global uniform bound $M$, this procedure yields a global solution on $[0,T]$ in a finite number of steps.
\end{proof}

\begin{cor}\label{cor:fSDE_global_general}
	Let $0<T<\infty$. Consider the fSDE
	\begin{equation}\label{eq:h54}
		dX_t=a(X_t)dt + \sum_{i=1}^{d}b^i(X_t)dW_tc^i(X_t)
	\end{equation}
	Assume that $a$, and the products $b^ic^i, i=1,\dots d$ are globally operator Lipschitz. Then \eqref{eq:h54} has a unique global solution on $[0,T]$, uniformly bounded in operator norm.
\end{cor}
\begin{proof}
	The general equation \eqref{eq:h54} can be handled by the quality $bdWc+cdWb=(b+c)dW(b+c)-bdWb-cdWc$. The statement follows from similar arguments as in the proof of \Cref{thm:fSDE_globalexistence} and \Cref{lem:fSDE_global_uniform_bound}.
\end{proof}

\section{Controlled Diffusion Processes}\label{sec:defoptimcontrolproblem}
The aim of this section is to define the notion of a free finite stochastic optimal control problem.
The forward equation is a fSDE \eqref{intro-freeSDE-diffform-orig}, which allows nonlinear drift and diffusion terms, where both may depend on the state and control variable. In case of optimal control, we extend the possibly nonlinear drift and diffusion functions in \eqref{intro-freeSDE-diffform-orig} by a control variable. This requires additional conditions on the drift and diffusion in order to ensure global existence of a unique solution of the forward equation. \\
We start by defining the notion of a control variable and extending free stochastic differential equations (fSDEs) to include control components, followed by a global existence theorem. Then, we establish conditions for the objective functional to characterize the free stochastic optimal control problem. Our primary objective is to derive a free analogue to the stochastic maximum principle. The formal definition of a free backward stochastic differential equation (fBSDE) is provided in \cref{sec:FreeMaxPrinc:BSDE}.

\subsection{Control Process}\label{sec:DiffProc:Control}
We first define the notion of a control process in the context of non-commutative variables.
\begin{defn}
	Let $0<T<\infty$. A control process $U=(U_t)_{t\in[0,T]}$ is a self-adjoint process valued $\mathcal{A}^{sa}$, adapted to the filtration $\mathbb{F}=(\mathcal{A}_t)_{t\in[0,T]}$ and strongly measurable with respect to the time variable $t\in[0,T]$.
\end{defn}
In \Cref{sec:setofcontrols} we define the set $\mathcal{U}\subseteq \mathcal{A}^{sa}$ of admissible controls. Condition \eqref{eq:Cond-for-a-and-b-for-controlled-existence} will ensure the existence of a global solution of the forward equation and imply a well defined cost function. Before we proceed, we need to consider the assumption on the drift and diffusion function of the underlying fSDE.

\subsection{Drift Function}\label{sec:DiffProc:Drift}
The state $(X_t)_{t\in[0,T]}\subseteq \mathcal{A}^{sa}$ of the system is modeled by a free stochastic differential equation (fSDE). When dealing with control problems, we extend the drift function by a control process and consider operator functions of type $a:\mathcal{A}\times \mathcal{A}\rightarrow \mathcal{A}$.  The properties of $a$ need to ensure a unique strong solution of the underlying free stochastic differential equation.

We assume a drift function $a:\mathcal{A}\times \mathcal{A}\rightarrow \mathcal{A}$, which has the following properties:
\begin{enumerate}
	\item[(a)] 	$a$ is a  continuous function, for which $a(\mathcal{A}^{sa}\times\mathcal{A}^{sa})\subseteq \mathcal{A}^{sa}$ (see \Cref{rem:selfadjointinrange}).
	\item[(b)] There is a constant $K>0$ such that
	$$
	\|a(X_1,U)-a(X_2,U)\|\leq K\|X_1-X_2\|
	$$
	for all $U\in\mathcal{A}$ and $X_1,X_2\in\mathcal{A}$.   
\end{enumerate} 
Note that these assumptions imply, that for all  $X,U\in\mathcal{A}$,
\begin{align*}
	\|a(X,U)\| &\leq \|a(0,U)\| + L_a\|X\|.
\end{align*} 
\begin{rem}\label{rem:selfadjointinrange} Consider e.g. the drift $a(X,U)=XU$, which is not allowed since it may violate the self-adjoint requirement. Instead, for linear drift, we need to consider $a(X,U)=XU+UX$. In the case of linear quadratic optimal control, the drift is $a(X,U)=AX+XA+BU+UB$ for fixed $A,B\in\mathcal{A}$.
\end{rem}
 \begin{rem}\label{rem:aisC1}
 	In 
 	\Cref{sec:generalizedhamilton}, we define the Hamiltonian $H$ for the optimal control problem. Operator Lipschitz functions are not necessarily $C^1$. For the Hamiltonian to be well defined, it is required that $a$ is $C^1(\R)$. For the global existence of a solution of a fSDE, it is sufficient that $a$ is continuous (and operator Lipschitz).
 \end{rem}
\subsection{Diffusion Functions}\label{sec:DiffProc:Diffusion}
We deal with multiple diffusion terms as shown in equation \eqref{intro-freeCSDE-diffform}. Therefore we need to extend the diffusion functions by the control variable, i.e. $b^i, c^i:\mathcal{A}\times \mathcal{A}\rightarrow\mathcal{A}$ for $i=1,\dots,d$. To ensure global existence of the solution of the underlying fSDE we need some additional properties. 
\begin{enumerate}
	\item[(c)] $b^i$ and $c^i$ are continuous functions and $b^i(\mathcal{A}^{sa}\times\mathcal{A}^{sa})\subseteq \mathcal{A}^{sa}$ resp. $c^i(\mathcal{A}^{sa}\times\mathcal{A}^{sa})\subseteq \mathcal{A}^{sa}$ for all $i=1,\dots d$.
	\item[(d)] There is a constant $L_{bc}>0$,  such that
	\begin{align*}
		\|b^i(X_1,U)c^i(X_1,U)-b^i(X_2,U)c^i(X_2,U)\| &\leq L_{bc}\|X_1-X_2\|.
	\end{align*}
	for all $U\in\mathcal{A}$ and $X_1,X_2\in\mathcal{A}$.
\end{enumerate}
These assumptions imply, that for all  $X,U\in\mathcal{A}$,
\begin{align*}
	\|b^i(X,U)c^i(X,U)\| &\leq \|b^i(0,U)c^i(0,U)\| + L_b\|X\|, i=1,\dots,d
\end{align*} 
 \begin{rem} By the same reasons as given in \eqref{rem:aisC1}, we will need $b^i, c^i$ to be $C^1$, in order to define the Hamiltonian. At this point continuous (and operator Lipschitz) is sufficient.
\end{rem}
\subsection{Set of Controls}\label{sec:setofcontrols}
Given drift functions $a$ and diffusion function $b^i, c^i, i=1,\dots d$ as defined by assumptions (a)-(d). 
\begin{defn}
	 For $X_0\in\mathcal{A}^{sa}$ the set   $\mathcal{U}(X_0)\subseteq \mathcal{A}^{sa}$ is defined as the set of controls $U=(U_t)_{t\in [0,T]}$, such that
	\begin{equation}\label{eq:Cond-for-a-and-b-for-controlled-existence}
		\int_0^T \|a(0,U_s)\|^2 + \|b^i(0,U_s)c^i(0,U_s)\|^2ds < \infty.
	\end{equation}
	When the context is clear we simply write $\mathcal{U}$ and hide the dependency on the initial value $X_0\in\mathcal{A}^{sa}$.
\end{defn}

\begin{rem}
	Note that the selection of the evaluation at value $ 0\in\mathcal{A}$ in \eqref{eq:Cond-for-a-and-b-for-controlled-existence} is arbitrary. Any fixed value can be taken. The dependency on $U$ requires additional constraints to keep drift and diffusion bounded. Condition \eqref{eq:Cond-for-a-and-b-for-controlled-existence} ensures global existence of a solution of the forward equation (\Cref{thm:fCSDE_globalexistence}) as well as a well defined functional objective (see \Cref{sec:DiffProc:FunObj}).
\end{rem}

\subsection{Controlled fSDE}\label{sec:DiffProc:SysDyn}
A local existence theorem for initial value problems of fSDE was given in \cite[Theorem 3.1]{Kargin}. Local existence (and uniqueness) was proven by a Picard iteration assuming locally operator Lipschitz drift and diffusions in operator norm. 
We extend the drift and diffusion coefficients in the fSDE by the control processes $(U_t)_{t\in [0,T]}$. In \cite{gangbo2025viscsol} an existence theorem was stated for controlled fSDEs, where the underlying fSDE was equipped with a constant diffusion coefficient.
\begin{defn} \label{def-freeItoProcess-controlled}
	Let the initial value $X_0$ be a self-adjoint element  in $\mathcal{A}^{sa}$ and $a,b^i,c^i:\mathcal{A}^{sa}\times\mathcal{A}^{sa}\rightarrow \mathcal{A}$ fulfilling conditions (a) to (d). Let $U=(U_t)_{t\in[0,T]}$ a control. We call 
	\begin{equation}\label{intro-freeCSDE-diffform}
		dX_t=a(X_t,U_t)dt+ \sum\limits_{i=1}^d b^i(X_t,U_t)dW_tc^i(X_t,U_t),~ t\geq 0
	\end{equation}
	a  free controlled stochastic differential equation (fCSDE).	
	A solution to \eqref{intro-freeCSDE-diffform} with initial condition $X_0\in\mathcal{A}^{sa}$ at $t=0$ is a process $(X_t)_{t\in[0,T]}$ with the following properties:
	\begin{enumerate}
		\item $X_0 \in \mathcal{A}_0^{sa}$
		\item $X_t\in\mathcal{A}_t^{sa}$ for all $t\in[0,T]$ and $U_t\in \mathcal{U}(X_0)$
		\item The equation
		\begin{equation}\label{intro-freeCSDE-inform}
			X_t=X_0 + \int_{0}^t a(X_s,U_s)ds + \sum\limits_{i=1}^d\int_{0}^tb^i(X_s,U_s)dW_sc^i(X_s,U_s)
		\end{equation} 
		is fulfilled for all $t\in[0,T]$ and $(U_t)_{t\in[0,T]}\in \mathcal{U}(X_0)$.
	\end{enumerate}
		The notation $X_t^{X_0}$ is used to express the dependency of the solution of the inital value at $t=0$.
\end{defn}

\begin{thm}\label{thm:fCSDE_globalexistence} Let $0<T<\infty$. Consider a fCSDE as defined in \eqref{intro-freeCSDE-diffform} on the finite interval $[0,T]$. Assume conditions (a) - (d) listed in  \Cref{sec:DiffProc:Drift} and \Cref{sec:DiffProc:Diffusion} on the drift function $a$ and diffusion functions $b^i, c^i$, $i=1,\dots d$. Additionally, let
\begin{equation*}
	\int_0^T \|a(0,U_s)\|^2 + \|b^i(0,U_s)c^i(0,U_s)\|^2ds < \infty.
\end{equation*}
Then for every $X_0 \in \mathcal{A}$, there is a unique solution $X_t$ of \eqref{intro-freeCSDE-diffform} on $[0,T]$. The solution $X_t$ is uniformly bounded in $L_\infty(\varphi)$, i.e. $\|X_t^{X_0}\|\leq M(X_0)<\infty$, where the constant $M=M(X_0)$ depends on the initial value $X_0\in\mathcal{A}$.
\end{thm}

\begin{proof}
	The proof is similar to the steps in \Cref{thm:fSDE_globalexistence} resp. \Cref{cor:fSDE_global_general}. The key point is to show a global uniform bound on the controlled fSDE on finite intervals. By similar arguments as in \Cref{cor:fSDE_global_general} we formulate the proof for a single diffusion term $b(X,U)dWb(X,U)$. Just as in the proof of \Cref{lem:fSDE_global_uniform_bound} we estimate
	\begin{align*}
		\|X_t\| \leq \underbrace{\|X_0\| + \int_0^t (\|a(0,U_s)\| + L_a\|X_s\|)ds}_{=C_0} + \underbrace{2\sqrt{2}\int_0^t (\|b^2(0,U_s)\| + L_b\|X_s\|)^2ds}_{=C_1}.
	\end{align*}
	Due to \eqref{eq:Cond-for-a-and-b-for-controlled-existence} we can find constants $\alpha_g,\beta_g>0$ such that $C_0$ and $C_1$ hold the following estimates. Let 
	\begin{align*}
		C_0^2&\leq4 \max \left(\|X_0\|^2, \int_0^t\left(\|a(0,U_s)\| + L_a\|X_s\|\right)^2ds\right)\\
		&\leq 8\max \left(\|X_0\|^2, \int_0^t\|a(0,U_s)\|^2+L_a^2\|X_s\|^2ds\right) \leq \int_0^t\alpha_g + \beta_g\|X_s\|^2ds,
	\end{align*}
	and 
	\begin{align*}
		C_1^2&\leq16 \int_0^t\left(\|b^2(0,U_s)\|^2 + L_b^2\|X_s\|^2\right)ds \leq \int_0^t\alpha_g + \beta_g\|X_s\|^2ds. 
	\end{align*}
	Gronwall yields the global bound on the solution. Global existence follows by the same arguments as in \Cref{cor:fSDE_global_general}.
\end{proof}

\subsection{Functional Objective}\label{sec:DiffProc:FunObj}

As in  the classical case, we consider running costs $f:\R\times \mathcal{A}\times \mathcal{U}\rightarrow \mathcal{A}$ and terminal costs $g:\mathcal{A}\rightarrow \mathcal{A}$.
\begin{enumerate}
	\item[(e)] Let $g\in C^1(\R)$ be concave. Note that due to \Cref{lem:concaveinA} we have $\varphi(g(Y)) \leq \varphi(f(X))+\varphi(g'(X)(Y-X))$, where $g'$ ist the derivative in $\R$.
	\item[(f)] Let $f:\R\times\mathcal{A}\times \mathcal{U}(X_0)\rightarrow\mathcal{A}$ be $C^1$, $f(\R\times \mathcal{A}^{sa}\times \mathcal{U})\subseteq \mathcal{A}^{sa}$ and 
	\begin{equation*}
		\|f(t,X,U)\| \leq \kappa(\|X\|) +  C_2(1+\|a(0,U)\|^2 + \|b^i(0,U)c^i(0,U)\|^2 )
	\end{equation*}
	for some $\gamma>0$, $C_1,C_2>0$ and some positive continuous  function $\kappa:\R\rightarrow \R^+_0$. Note that the solution $(X_t)_{t\in[0,T]}$ of our forward equation is uniformly bounded.
\end{enumerate}
Since the solution is uniformly bounded with $M(X_0)<\infty$ the condition \eqref{eq:Cond-for-a-and-b-for-controlled-existence} ensures, that for ever $X_0\in \mathcal{A}^{sa}$
\begin{equation}
	\varphi\left( \left| \int_{0}^T f(s,X_s^{X_0},U_s)ds\right| \right)<\infty.
\end{equation}
This implies, that $\mathcal{U}(X_0)$ does not depend on the initial value $X_0$. We write $\mathcal{U}$ in the following. Thus we can then define the cost associated to the control $U\in\mathcal{U}$ as  
\begin{equation}\label{eq:CostFunction}
	J(U)=\varphi \left( \int_0^T f(s, X_s^{X_0}, U_s)ds + g(X_T^{X_0})\right).
\end{equation}
\subsection{Stochastic Optimal Control Problem}\label{sec:DiffProc:StochOptimControl}
Let $T>0$ and $X_0\in\mathcal{A}_0$. Given the fSDE \eqref{intro-freeCSDE-inform}. Under conditions (a)-(f), the objective is to maximize the gain function $J$ over the set of admissible control processes $\mathcal{U}$, i.e.
\begin{equation}\label{eq:supcondition}
	\sup_{U\in\mathcal{U}}J(U).
\end{equation}
 Given an initial condition $X_0 \in \mathcal{A}_0$, we say that $\hat{U}$ is an optimal control if $J(\hat{U})=\sup_{U\in\mathcal{U}}J(U)$.
 
\section{Free Stochastic Maximum Principle}\label{sec:BSDEsandFreeOptimControl}
In this section we formulate the stochastic maximum principle to the optimal control problem defined in \Cref{sec:DiffProc:StochOptimControl}. First, we define in \Cref{sec:FreeMaxPrinc:BSDE}  the notion of a backward stochastic differential equation  and give a global existence theorem. In \Cref{sec:FreeStochMax_general} we develop the maximum principle, based on \cite[Theorem 6.4.6]{Pham2009}. Free backward stochastic differential equations already appeared in \cite{das2025freeprobabilisticframeworkdenoising} and \cite{dabrowski2016free}.

\subsection{Free Backward Differential Equation}\label{sec:FreeMaxPrinc:BSDE}

As usual, consider a free Brownian motion $(W_t)_{t\geq 0}$ adapted to the filtration $\mathbb{F}=(\mathcal{A}_t)_{t\geq 0}$. The definition of a free backward SDE relies on the martingale representation theorem \cite[Proposition 5.3.13]{Biane1998}. In the proof of \Cref{T:BSD1}, we construct a martingale $M_t$ which is expressed as a free stochastic integral in terms of so called biprocess $(Z_t)_{t\in[0,T]}$ (see \eqref{E:BSDE4} and \Cref{sec:freeCalc-BiprocStochInt}). The biprocess $(Z_t)_{t\in[0,T]}$ is valued in $L_2(\varphi \otimes \varphi^{op})$, we refer to \cite{Biane1998} and  \Cref{sec:freeCalc-BiprocStochInt}.\\

 The set  $L_2(\varphi)^{sa}\subset L_2(\varphi)$ contains all self-adjoint $L_2(\varphi)$ operators.
$\mathbb{L}_2(\varphi)^{sa}$ denotes the set of self-adjoint, strongly measurable processes $(Y_t)_{t\in[0,T]}$, valued in $L_2(\varphi)^{sa}$ and adapted to $\mathbb{F}$. Define $\mathcal{H}=\mathbb{L}_2(\varphi)^{sa} \times \mathcal{B}_2^{a}$ endowed with the norm 
$$\|(Y,Z)\|_S=\Bigl(\int_0^Te^{\alpha s}(\|Y_s\|_2^2+\|Z_s\|_{L_2(\varphi\otimes\varphi^{op})}^2)ds\Bigr)^{\frac1{2}}.$$
The scalar product on $\mathcal{H}$ is given by
$$
\langle (U,V),(Y,Z)\rangle_\mathcal{H}=\int_0^T e^{\alpha s}\langle U_s,Y_s \rangle_{L_2(\varphi)}ds + \langle V,Z \rangle_{\mathcal{B}_2^a}.
$$
The scalar product
in $B_2^a$ by $\langle U,V \rangle =\int_0^T \langle U_s,V_s \rangle_{L_2(\varphi)}ds$, see \cite{Biane1998}.\\

 Let $T>0$ and  $\xi\in \mathcal{A}^{\text{sa}}$.   Assume, that the function  $k:[0,T]\times L_2(\varphi)\times L_2(\varphi \otimes \varphi^{op})\longrightarrow L_2(\varphi)$ has the following properties:
\begin{itemize}
	\item $k$ is continuous and operator Lipschitz, s.t. there exists a $C_k>0$
	$$\|k(t,U_1,V_1)-k(t,U_2,V_2)\|_2\le C_k\left(\|U_1-U_2\|_2 + \|V_1-V_2\|_{L_2(\varphi\otimes\varphi^{op})}\right)$$ for all $U_1,U_2\in L_2(\varphi)$,  $V_1,V_2\in L_2(\varphi \otimes \varphi^{op})$ and $t\in[0,T]$
	\item $\varphi \left(\int_0^Tk(s,0,0)^2ds\right)<\infty$
	\item $k([0,T]\times L_2(\varphi)^{sa} \times L_2(\varphi\otimes\varphi^{op})^{sa}) \subseteq L_2(\varphi)^{sa}$
\end{itemize}
\begin{defn}\label{def:freeBSDE} 
	A free backward stochastic differential equation (fBSDE) on $[0,T]$  is a time reversed free stochastic differential equation 
	\begin{equation}\label{E:BSDE1}
		-dY_t=k(t,Y_t,Z_t)dt-Z_t\#dW_t,\ Y_T=\xi.
	\end{equation}
	$Z=(Z_t)_{t\in[0,T]}$ is a self-adjoint, adapted biprocess from $\mathcal{B}_2^{a}$,  $Y\in\mathbb{L}_2(\varphi)^{sa}$ and $k$ as defined above.
\end{defn}
As we will see in the proof of \Cref{thm:freestochprinciple-1}, the martingale representation theorem \cite[Theorem 5.3.13]{Biane1998} determines the diffusion of the backward differential equation.
We want to proof the existence of a free backward SDE on a finite interval $[0,T]$.

As in the classical BSDE we define the solution of (\ref{E:BSDE1}). 

\begin{defn}\label{D:BSDE1}
	A solution of the free BSDE (\ref{E:BSDE1}) is a pair $(Y,Z)\in \mathcal{H}$, s.t. for all $t\in[0,T]$
	\begin{equation*}
	Y_t=\xi+{\int_t^Tk(s,Y_s,Z_s)ds}-\int_t^TZ_s\#dW_s
	\end{equation*}
\end{defn}
With fixing the appropriate preliminaries we prove
\begin{thm}\label{T:BSD1}
	If the pair $(\xi,k)$ satisfies the conditions in \Cref{def:freeBSDE}, then there exists exactly one solution $(Y,Z)\in\mathcal{H}$ of the free BSDE (\ref{E:BSDE1}).
\end{thm}
\begin{proof} We follow mainly the proof in \cite{Pham2009}, which is a Picard argument with an appropriate norm. 

	To prove the existence of a solution on $[0,T]$, we apply a fixpoint argument on the Hilbert space $\mathcal H$. For this we select $(U,V)\in \mathcal H$. Then we define $\Psi(U,V)=(Y,Z)$ as follows. For $t\in[0,T]$ let
	$$M_t=\mathbb E_{\mathcal F_t}(\xi+\int_0^Tk(s,U_s,V_s)ds).$$
	Due to the assumptions on $k$,  $M_t$ is well defined, self-adjoint and is a martingale in $L_2(\varphi)$. The martingale representation theorem \cite[Proposition 5.3.13]{Biane1998} shows, that there exists a $Z\in \mathcal B_2^a$ and $M_0\in\R$ s.t.
	\begin{equation}\label{E:BSDE2}
		M_t=M_0+\int_0^tZ_s\#dW_s.
	\end{equation}
	Then the integral $\int_0^t Z_s\#dW_s$ is self-adjoint, which implies, that $Z_s$ is self-adjoint (see the remark after \cite[Definition 2.2.1]{Biane1998}).
	Now define
	\begin{equation}\label{E:BSDE3}
		Y_t:=\mathbb E_{\mathcal F_t}(\xi+\int_t^Tk(s,U_s,V_s)ds)=M_t-\int_0^tk(s,U_s,V_s)ds
	\end{equation}
	Since $Y_T=\xi$ by definition and the representation of $(M_t)$  in (\ref{E:BSDE2}) we realize that
	\begin{equation}\label{E:BSDE4}
		Y_t=\xi+\int_t^Tk(s,U_s,V_s)ds-\int_t^TZ_s\#dW_s.
	\end{equation}
	
	Since $\xi \in \mathcal{A}^{sa}$ we have $\|\xi\|<\infty$.
	Due to the It\^{o} isometry of the free stochastic integral in $L_2(\varphi)$ (see \cite[Corollary 3.1.2]{Biane1998}), the assumptions on $k$ and $Z_s\in\mathcal B_2^a$ it follows  that  $Y\in \mathcal H$. We see that $(Y,Z)$ is a solution to (\ref{E:BSDE1}) if $(Y,Z)$ is fixed point to $\Psi$.\\
	Now let $(U,V),(U',V')\in \mathcal H$ and 
	$(Y,Z)=\Psi((U,V))$, $(Y',Z')=\Psi((U',V'))$ and $(\tilde{Y},\tilde{Z})=(Y-Y',Z-Z')$. In addition, $\tilde{k_t}=k(t,U_t,V_t)-k(t,U_t',V_t')$.
	
	Now let $\alpha >0$ and consider $s\longmapsto e^{\alpha s}(\tilde{Y}_s^2)$. Since the It\^{o} formula in product form is also valid in $L_2(\varphi)$ (\cite[Corollary 3.1.2]{Biane1998}), we obtain
	\begin{eqnarray*}
		d(Y_s^2)&=&(Y_s+dYs)(Y_s+dY_s)-Y_s^2\\
		&=&Y_sdY_s+(dY_s)Y_s+\varphi(Z_s^2)dt
	\end{eqnarray*}
	Hence
	$$d(e^{\alpha s}(Y_s)^2)=\alpha e^{\alpha s}(Y_s)^2+e^{\alpha s}\Bigl(Y_sdY_s+(dY_s)Y_s+
	\varphi(Z_s^2)dt\Bigr)$$
	Hence from \eqref{E:BSDE4}
	\begin{multline*}
		(\tilde{Y}_0)^2=-\int_0^Te^{\alpha s}\Bigl(\alpha (\tilde{Y}_s^2)-(\tilde{k}_s\tilde{Y}_s+\tilde{Y}_s\tilde{k}_s)\Bigr)ds-\int_0^Te^{\alpha s}\varphi(Z_s^2) ds
		-\\ -
		\int_0^Te^{\alpha s}\Bigl(\tilde{Y}_s\tilde{Z_s}\#dW_s+\tilde{Z_s}\#dW_s\tilde{Y}_s\Bigr)
	\end{multline*}
	Now
	$$N_t=\int_0^te^{\alpha s}\Bigl(\tilde{Y}_s\tilde{Z_s}\#dW_s+\tilde{Z_s}\#dW_s	\tilde{Y}_s\Bigr)$$
	is a bounded martingale in $\mathcal H$. Applying the trace to the equation above, we obtain
	\begin{eqnarray*}
		&&\varphi(\tilde{Y}_0^2)+\Bigl(\int_0^Te^{\alpha s}(\alpha \|\tilde{Y}_s\|_2^2+\|\tilde{Z}_s\|_{\mathcal B_2}^2ds\Bigr)
		=\varphi\Bigl(\int_0^Te^{\alpha s}(\tilde{k}_s\tilde{Y}_s+\tilde{Y}_s\tilde{k}_s)ds\Bigr)\\
		&\le&2C_k\int_0^Te^{\alpha s}\|\tilde{Y}_s\|_2\Bigl(\|\tilde{U}\|_2+\|\tilde{V}_s\|_{\mathcal B_2}\Bigr)ds\\
		&\le&4C_k^2\Bigl(\int_0^Te^{\alpha s}\|\tilde{Y}_s\|^2_2ds\Bigr)+\frac1{2}\Bigl(\int_0^Te^{\alpha s}(\|\tilde{U}_s\|_2^2+\|\tilde{V}_s\|_{\mathcal B_2}^2)ds\Bigr)
	\end{eqnarray*}
	where the last inequality results by applying the binomial formula on the integrand.
	Define $\alpha=1+4C_k^2$ then $\Phi$ is a contradiction, 
	\begin{equation*}
		\|(\tilde{Y},\tilde{Z})\|_S \leq \frac{1}{2}\|(\tilde{U},\tilde{V})\|_S.
	\end{equation*}
\end{proof}

\subsection{Free Stochastic Maximum Principle}\label{sec:FreeStochMax_general}
The difficulty in formulating a maximum principle in the free probability setting based on \Cref{thm:pontryagin-classic} is the notion of the generalized Hamiltonian $H$ (or shortly Hamiltonian in the following) and it's derivatives $H$ w.r.t to the state and control variable. In general, since the state of our system is valued in $\mathcal{A}$, we must expect $H$ to be an operator function. The study of derivatives of operator functions is a classical topic, see \cite[Chapter 5]{Skripka2019}. Derivatives are expressed in terms of so called multiple  operator integrals (MOI), which have a widespread use. It would be natural, to attempt to formulate the optimality condition in  \Cref{thm:pontryagin-classic} by help of MOIs. Nevertheless, it is expected, that the optimality condition results in a framework, which is difficult to apply, especially if one considers concrete examples and applications. Therefore our aim is to stay as much as possible in $\R$, that is defining a function $H:\R\rightarrow \R$ and it's derivatives such, that a maximum principle more or less similar to \Cref{thm:pontryagin-classic} will be obtained. In this section we show, that it is indeed possible to find a real valued function $H$, such that $H$ has  the necessary properties to obtain a free analog of the maximum principle. The overall attempt is aligned to the commutative case in it's nature, but two difficulties arise. First, the variable $Z$ in the diffusion of the backward differential equation (BSDE) is taken from the space of biprocesses $\mathcal{B}_2^a$, which means that the diffusion does not really have a nice representation.  Second, the It\^{o}-formula causes the Hamiltonian to be more complicated compared to the commutative case and difficult to work with. Nevertheless, we present linear examples, for which a explicit solution can be formulated.\\
Consider an optimal control problem as defined in \eqref{sec:DiffProc:StochOptimControl}. The forward equation is given by  \eqref{intro-freeCSDE-diffform}, i.e.
\begin{equation}
	dX_t=a(X_t,U_t)dt+ \sum\limits_{i=1}^d b^i(X_t,U_t)dW_tc^i(X_t,U_t),
\end{equation}
with $U\in\mathcal{U}$.\\
Before we start, we need the following important Lemma.
\begin{lem}\label{lem:concaveinA}
	Let $f\in C^2(\R)$ be concave. Let $X,Y\in\mathcal{A}^{sa}$. Then
	\begin{equation}
		\varphi(f(Y))\leq \varphi(f(X))+\varphi(f'(X)(Y-X)).
	\end{equation}
\end{lem}
\begin{proof}
Since $f$ is concave, the function $g(y)=f(y)-f(x)-f'(x)(y-x)\leq 0$ for arbitrary $x$. By functional calculus $g(Y)$ is a negative operator, therefore $\varphi(g(Y))\leq 0$. 
\end{proof}

\subsubsection{Hamiltonian function $H$}\label{sec:generalizedhamilton}

The main theorem is \Cref{thm:freestochprinciple-1}. As we will see, the proof literally follows the classical proof of \cite[Theorem 6.4.6]{Pham2009}. The difference lies in the It\^{o}-formula (\Cref{sec:freeItoFormula}), which has to be applied to the product of the solution $(X_t)$ of the forward equation and the solution process of the BSDE $(Y_t)$. This requires to calculate the  products of diffusion terms of the BSDE and the forward fSDE. Since the diffusion of $Y_t$ consists of a general biprocess $Z_t\in\mathcal{B}_2^a$, we use the representation of the adapted processes in the stochastic integrals on the full Fock space, as described in \Cref{sec:freeCalc-BiprocStochInt}, see also \cite[Chapter 2]{kummererspeicher}.\\

Let $Z=\sum_{j\in\Z}A_j\otimes B_j\in\mathcal{B}_2^a$.  
Note that $Z$ comes from the martingale representation of the self-adjoint process $(M_t)_{t\geq 0}$ in \eqref{E:BSDE2}. As described in the part after introducing \eqref{E:BSDE2} by the martingale representation theorem, $Z$ is self-adjoint. The It\^{o}-formula applied to the product $X_t Y_t$ (resp. $X_tY_t$) requires to calculate the product of the diffusion terms of the processes $(X_t)$ and $(Y_t)$. This product will appear in the generalized Hamiltonian as defined below. The It\^{o}-formula gives (using the abbreviation $b_t^i=b^i(X_t,U_t)$ )
\begin{equation}\label{sec:quadrvariation_ito_hamiltonian-l}
	\sum_{i=1}^d (b^i_tdW_t c^i_t) Z_t\# dW_t=
	 \sum_{i=1}^d (b^i_tdW_t c^i_t) \sum_{j\in\Z}A_{j,t}dW_t B_{j,t} = \sum_{i=1}^d\sum_{j \in\Z}b_t^i B_{j,t}\varphi(c_t^i A_{j,t})dt
\end{equation}
and
\begin{equation}\label{sec:quadrvariation_ito_hamiltonian-r}
	Z_t\# dW_t\sum_{i=1}^d (b^i_tdW_t c^i_t) =
	\sum_{j\in\Z}A_{j,t}dW_t B_{j,t}\sum_{i=1}^d (b^i_tdW_t c^i_t)  = \sum_{i=1}^d\sum_{j \in\Z}A_{j,t} c_t^i\varphi(b_t^i B_{j,t})dt
\end{equation}
\textbf{Definition of the Hamiltonian over $\R$:}
Consider $z\in \ell_2(\Z)\otimes \ell_2(\Z)$, i.e. $z=\sum_{j\in\Z}\alpha_j\otimes \beta_j$. Define $H:\mathbb{T}\times\R\times\R\times\R\times \ell_2(\Z)\otimes\ell_2(\Z)\rightarrow \R$ motivated by \eqref{sec:quadrvariation_ito_hamiltonian-l} resp. \eqref{sec:quadrvariation_ito_hamiltonian-r} as
\begin{equation}\label{eq:genHamilton_inR}
	H(t,x,u,y,z)=
	1/2(a(x,u)y+ya(x,u))+ \Sigma(x,u,z)  + f(t,x,u).
\end{equation}
$\Sigma$ is given by $\Sigma(x,u,z)=1/2(\Sigma_l(x,u,z) + \Sigma_r(x,u,z))$ where 
\begin{equation*}
\begin{aligned}
	\Sigma_l(x,u,z)&=\sum_{i=1}^d \sum_{j\in\Z}b^i(x,u)\beta_j\varphi(c^i(x,u) \alpha_j \cdot \textbf{1})
\end{aligned}
\end{equation*}
and
\begin{equation*}
	\begin{aligned}
		\Sigma_r(x,u,z)&=\sum_{i=1}^d \sum_{j\in\Z}\alpha_jc^i(x,u)_j\varphi(b^i(x,u) \beta_j \cdot \textbf{1})
	\end{aligned}
\end{equation*}
Since $\alpha_j, \beta_j\in \ell_2(\Z)$ the sum over $j$ converges due to the Cauchy-Schwarz inequality. Consider the  drift $a\in C^1(\R)$ and diffusion functions $b^i, c^i\in C^1(\R^2)$. Note that for the existence of a global solution of the fSDE we only require $a,b^i,c^i$ to be continuous and operator Lipschitz, which does not imply that they are $C^1$ (see \cite{AleksandrovPeller}, \cite{Skripka2019}).
Taking the derivative in $\R$ gives
\begin{equation}\label{eq:genHamilton_deriv_inR}
\begin{aligned}
		H_x(t,x,u,y,z) =
	1/2(a_x(x,u)y+ya_x(x,u)) + \Sigma_x(x,u,z)  + f_x(t,x,u),
\end{aligned}
\end{equation}
where $\Sigma_x=1/2(\Sigma_{l,x}+\Sigma_{r,x})$ and
\begin{equation*}
	\begin{aligned}
		\Sigma_{l,x}(x,u,z) &= 		
		\sum_{i=1}^d\sum_{j\in\Z} b_x^i(x,u) \beta_j\varphi(c^i(x,u) \alpha_j \cdot \textbf{1})
		+\\
		&+
		\sum_{i=1}^d\sum_{j\in\Z}b^i(x,u) \beta_j\varphi( c_x^i(x,u) \alpha_j \cdot \textbf{1}),
	\end{aligned}
\end{equation*}
and
\begin{equation*}
	\begin{aligned}
		\Sigma_{r,x}(x,u,z) &= 		
		\sum_{i=1}^d \sum_{j\in\Z}\alpha_jc_x^i(x,u)_j\varphi(b^i(x,u) \beta_j \cdot \textbf{1})
		+\\
		&+
		\sum_{i=1}^d \sum_{j\in\Z}\alpha_jc^i(x,u)_j\varphi(b_x^i(x,u) \beta_j \cdot \textbf{1}).
	\end{aligned}
\end{equation*}

\textbf{Definition of the Hamiltonian over $\mathcal{A}$:} Let $X,U,Y,Z\in\mathcal{A}^{sa}$. 
We apply the notation $\{X,Y\}=1/2(XY+YX)$ for better readability. Note that $\varphi(\{X,Y\})=\varphi(XY)$. As mentioned before, Consider the  drift $a\in C^1(\R)$ and diffusion functions $b^i, c^i\in C^1(\R^2)$.\\
Under application of the functional calculus we consider  $H:\R\times\mathcal{A}\times\mathcal{A}\times\mathcal{A}\times\mathcal{A}\rightarrow \mathcal{A}$ given by
\begin{equation}\label{eq:genHamilton_inA}
	H(t,X,U,Y,Z)=
	\{a(X,U),Y\}+ \Sigma(X,U,Z) + f(t,X,U),
\end{equation}
where $\Sigma(X,U,Z)=1/2(\Sigma_l(X,U,Z)+\Sigma_r(X,U,Z))$ and
\begin{equation*}
	\Sigma_l(X,U,Z)=\sum_{i=1}^d\sum_{j\in\Z}b^i(X,U) B_j\varphi(c^i(X,U) A_j)
\end{equation*}
and
\begin{equation*}
		\Sigma_r(x,u,z)=\sum_{i=1}^d \sum_{j\in\Z}A_jc^i(X,U)\varphi(b^i(X,U) B_j).
\end{equation*}
Due to the definition of $\Sigma$ and the assumptions on $f$ (condition (f) in \Cref{sec:DiffProc:FunObj}), the Hamiltonian $H$ fulfills $H(\R\times\mathcal{A}^{sa}\times\mathcal{A}^{sa}\times\mathcal{A}^{sa}\times\mathcal{A}^{sa})\subseteq \mathcal{A}^{sa}$. Taking the real derivative w.r.t to $X$ yields \begin{equation}\label{eq:genHamilton_deriv_inA}
	H_x(t,X,U,Y,Z)=
	\{a_x(X,U),Y\}+ \Sigma_x(X,U,Z) + f_x(t,X,U), 
\end{equation}
where $\Sigma_x(X,U,Z)=1/2(\Sigma_{l,x}(X,U,Z)+\Sigma_{r,x}(X,U,Z))$ and
\begin{equation*}
	\begin{aligned}
		\Sigma_{l,x}(X,U,Z)&=\sum_{i=1}^d\sum_{j\in\Z} b_x^i(X,U) B_j\varphi(c^i(X,U) A_j)+ \\
		&+\sum_{i=1}^d\sum_{j\in\Z}b^i(X,U) B_j\varphi(c_x^i(X,U) A_j)
	\end{aligned}
\end{equation*}
and
\begin{equation*}
	\begin{aligned}
		\Sigma_{r,x}(x,u,z) &= 		
		\sum_{i=1}^d \sum_{j\in\Z}A_jc_x^i(X,U)\varphi(b^i(X,U)B_j)
		+\\
		&+
		\sum_{i=1}^d \sum_{j\in\Z}A_jc^i(X,U)\varphi(b_x^i(X,U) B_j).
	\end{aligned}
\end{equation*}
In many cases, the underlying fSDE consist of diffusion terms either as a single summand, or double summand. 
\begin{itemize}
	\item If the diffusion term of \eqref{intro-freeCSDE-diffform} is given by  $b(X_t, U_t)dW_t+dW_tb(X_t, U_t)$, then \eqref{eq:genHamilton_inA} simplifies to 
	\begin{align*}
		\Sigma_l(X,U,Z)&=\sum_{j\in\Z}b(X,U)B_j\varphi(A_j)+B_j\varphi(b(X,U)A_j)\\
		\Sigma_r(X,U,Z)&=\sum_{j\in\Z}A_j\varphi(b(X,U)B_j)+A_jb(X,U)\varphi(B_j)
	\end{align*}

	\item If the diffusion functions $b^i, c^i$ do not depend on $X$, then $\Sigma=\Sigma(U,Z)$.
	\item If the diffusion functions $b^i, c^i$ do not depend on $U$, then $\Sigma=\Sigma(X,Z)$.
	\item If the diffusion functions $b^i, c^i$ do not depend on $X$ and $U$, i.e. are constant, then $\Sigma=\Sigma(Z)$.	
\end{itemize}
As a consequence of \Cref{lem:concaveinA}, we have the following result, which is crucial for the proof of \cref{thm:freestochprinciple-1}.
\begin{lem}
	Consider the generalized Hamiltonian over $\R$ defined in \eqref{eq:genHamilton_inR}. If $(x,u)\rightarrow H(x,u,y,z)$ is concave, then
	\begin{equation*}
		\varphi\left(H(t,\hX,\hU_t,\hY_t, \hZ_t)-
		H(t,X_t,U_t,\hY_t, \hZ_t)-H_x(t,\hX_t,\hU_t,\hY_t, \hZ_t)(\hX_t-X_t)\right)\geq 0 
	\end{equation*}
	for all $t\in[0,T]$, $\hat{X}_t, X_t, Y_t \in \mathcal{A}$, $Z_t\in\mathcal{B}_2^a$ and $\hat{U}_t,U_t\in\mathcal{U}$.
\end{lem}

\begin{defn}\label{def:BSDE}
	Let $t\in[0,T]$ and $Z\in \mathbb{B}_2^a$. Consider the Hamiltonian $H$ given in \eqref{eq:genHamilton_inA} and the derivative $H_x$ in  \eqref{eq:genHamilton_deriv_inA}. For each $U\in\mathcal{U}$, the backward SDE (BSDE)
	\begin{equation}\label{eq:BSDE}
		-dY_t=H_x(t,X_t,U_t,Y_t,Z_t)dt-Z_t\#dW_t,~Y_T=g_x(X_T)
	\end{equation}
	is called the adjoint equation of the optimal control problem defined in \Cref{sec:defoptimcontrolproblem}.
\end{defn}

Now we are in the position to formulate the stochastic version of the stochastic maximum principle in the free probability setting.
\begin{thm}\label{thm:freestochprinciple-1}
	Consider the fSDE \eqref{intro-freeCSDE-inform}, BSDE \eqref{eq:BSDE} and cost functional as defined in \Cref{sec:DiffProc:FunObj}. Let $\hat{U}\in\mathcal{U}$ and $\hat{X}$ the associated controlled diffusion. Assume $(\hat{Y},\hat{Z})$ is a solution of \eqref{eq:BSDE} such that
	\begin{equation}\label{eq:optt_cond_Pontryagin}
			\varphi \left(H(t,\hX_t,\hat{U}_t,\hY_t,\hZ_t)\right)
			=
			\max_{U\in \mathcal{U}}\varphi \left(H(t,\hX_t,U,\hY_t,\hZ_t)\right).
	\end{equation}
	If $(x,u)\rightarrow H(x,u,y,z)$  is concave, then $\hat{U}$ is an optimal control, i.e.
	\begin{equation*}
		J(\hat{U})=\sup_{U\in \mathcal{U}}J(U).
	\end{equation*}
\end{thm}

\begin{proof} We follow the proof of \cite[Theorem 6.4.6]{Pham2009} and start to express
	\begin{equation}\label{eq:tmp1}
		J(\hU)-J(U)=\varphi\left[\int_0^Tf(t,\hX_t,\hU_t)-f(t,X_t,U_t)dt +g(\hX_T)-g(X_T)\right].
	\end{equation}
	Since $g$ is concave \Cref{lem:concaveinA} gives
	\begin{align*}
		\varphi\left(g(\hX_T)-g(X_T)\right)& \geq \varphi\left(g_x(\hX_T)(\hX_T-X_T)\right) 
		=\varphi\left(\hY_T(\hX_T-X_T)\right)\\
		=&\varphi\left(\int_0^Td\hY_t(\hX_t-X_t) + \int_0^T\hY_t (d\hX_t-dX_t)\right)\\
		&+\varphi\left( \int_0^T d\hat{Y}_t(d\hat{X}_t-dX_t) \right).
	\end{align*}
	Inserting the BSDE \eqref{eq:BSDE} yields
	\begin{multline}\label{eq:tmp2}
		\varphi\left(g(\hX_T)-g(X_T)\right)=\varphi\left(\int_0^T -H_x(t,\hX_t,\hU_t,\hY_t,\hZ_t)(\hX_t-X_t)dt + \int_0^t \hY_t(\hat{b}_t-b_t)dt\right)\\
		+\varphi\left( \int_0^T d\hat{Y}_t(d\hat{X}_t-dX_t)\right).
	\end{multline}

	By the definition of $H$ in \eqref{eq:genHamilton_inA} we obtain

	\begin{equation}\label{eq:tmp3}
		\begin{aligned}
			\varphi\left[\int_0^Tf(t,\hX_t,\hU_t)-f(t,X_t,U_t)dt\right] 
			&= \varphi \left( \int_0^T H(\hX_t,\hU_t,\hY_t,\hZ_t)-H(X_t,U_t,\hY_t,\hZ_t)dt\right)\\
			&- \int_0^T\varphi((a(\hX_t,\hU_t)-a(X_t,U_t))Y_t)dt\\
			&- \int_0^T \varphi(\Sigma (\hat{X}_t,\hat{U}_t,\hat{Y}_t,\hat{Z}_t))dt\\
			&- \int_0^T \varphi(\Sigma (X_t,U_t,\hat{Y}_t,\hat{Z}_t))dt.
		\end{aligned}
	\end{equation}
Inserting \eqref{eq:tmp2}, \eqref{eq:tmp3} into \eqref{eq:tmp1} give 
\begin{align*}
	J(\hat{U})-J(U)&\geq\varphi\left(H(t,\hX,\hU_t,\hY_t, \hZ_t)\right)-\varphi\left(
	H(t,X_t,U_t,\hY_t, \hZ_t)\right)-\\
	&-\varphi\left(H_x(t,\hX_t,\hU_t,\hY_t, \hZ_t)(\hX_t-X_t)\right) .
\end{align*}
Since $H$ is concave w.r.t to the argument $x$ and $u$, applying Lemma (\ref{lem:concaveinA}) shows that  $J(\hat{U})-J(U)\geq 0$.
\end{proof}
\begin{rem}\label{rem:h_u_optim_remark}
	The optimality condition in \eqref{eq:optt_cond_Pontryagin} can be handeled by the the derivative of the Hamilonian w.r.t. $u$ in $\R$. 
	If we calculate the G\^{a}teaux derivative of $H$ in the von Neumann algebra along linear path of self-adjoints, then
	\begin{equation*}
		\varphi\left(\frac{d}{d\epsilon} \varphi(H(X,U+\epsilon V,Y,Z))\right)= \varphi\left(T_H^{U,U}(V)\right)=\varphi\left(H_u(X,U,Y,Z)V\right)
	\end{equation*}
	for all $V\in\mathcal{A}^{sa}$. See \cite[Theorem 5.3.5]{Skripka2019}. The definition of the double operator integral $T_H^{U,U}(V)$ can also be found in this reference.\\
	Then \eqref{eq:optt_cond_Pontryagin} can be handeled by first calculating  the extreme values in $\R$ and then carrying over the result to the non-commutative context. In the next section, we will make use of this procedure to obtain solutions to optimal control problems.
\end{rem}

\section{Examples}\label{sec:examples}
The aim of the following examples is to find an explicit solution of the underlying optimal control problem. The solution strategy is aligned by the strategies of solving a commutative control problem. We solely consider linear problems.
\subsection{Example 1}
It is taken from \cite{gangbo2025viscsol}. Given a fSDE by
\begin{equation*}\label{eq:Example1-fSDE}
	dX_t=U_tdt + dW_t, X_0\in\mathcal{A}_0^{sa}, 
\end{equation*}
and a cost functional
\begin{equation*}
	J(U)=-\varphi\left(\int_0^T1/2U_s^2ds + \theta_1X_T^2+\theta_2\varphi (X_T)X_T \right),
\end{equation*}
where $\theta_1,\theta_2\in\R$.
Our aim is to determine $\sup_{U\in\mathbb{U}} J(U)$.
Note that $g$ in $\R$ is given by $g(x)=(\theta_1+\theta_2)x^2$. We can rewrite $g$ as $g(x)=\theta_1x^2+\theta_2\varphi(x\text{id})x$, which gives the expression in the cost functional $J$ by applying the functional calculus. The end condition in the adjoint equation requires $g_x=2(\theta_1+\theta_2)x$.
The generalized Hamiltonian \eqref{eq:genHamilton_inR}
\begin{equation*}
	H(t,x,u,y,z)=1/2(uy+yu)+ 2\varphi(z\text{id}) + 2z - 1/2u^2.
\end{equation*}
The optimality condition \eqref{eq:optt_cond_Pontryagin} can be handled by \Cref{rem:h_u_optim_remark}. Since in $\R$ we get $H_u(x,u,y,z)=u-y$, the maximum condition gives $\hat{U}_t=\hat{Y}_t$.
The adjoint equation results in
\begin{equation}\label{eq:Example1-BSDE}
	dY_t=Z_t\#dW_t, \, Y_T=-2(\theta_1+\theta_2)X_T.
\end{equation}
We search a function $p\in C^1([0,T],\R)$, such that  
\begin{equation}\label{eq:Example1-AnsatzYP}
	\hat{Y}_t = p(t)\hat{X}_t.
\end{equation}
We simply write $p$ and skip the dependency on $t$ when the context is clear.  By differentiating \eqref{eq:Example1-AnsatzYP} and inserting \eqref{eq:Example1-fSDE} yields 
\begin{align*}
	d\hat{Y}_t&=dp\hat{X}_t+pd\hat{X}_t\\
	&=dp\hat{X}_t+p\hat{U}_tdt + pdW_t\\
	&=dp\hat{X}_t+p^2\hat{X}_tdt + pdW_t
\end{align*}
Comparing the drift terms with \eqref{eq:Example1-BSDE} and using $\hat{U}_t=\hat{Y}_t=p\hat{X}_t$, we obtain the deterministic differential equation
\begin{equation}\label{eq:diffeqPexampleJekel}
	\dot{p}=-p^2, p(T)=-2(\theta_1+\theta_2).
\end{equation}
In the deterministic context, this differential equation is the Riccati differential equation for the given linear quadratic optimal control problem. 
The solution of \eqref{eq:diffeqPexampleJekel} is
\begin{equation}
	p(t)=\left((-1/(2(\theta_1+\theta_2))+(t-T))\right)^{-1}.
\end{equation}
Thus
\begin{equation}
	\hat{U}_t=\left((-1/(2(\theta_1+\theta_2))+(t-T))\right)^{-1}\hat{X}_t,
\end{equation}
and the optimal process is an Ornstein-Uhlenbeck process given as a solution of the fSDE
\begin{equation}
	d\hat{X}_t=\left((-1/2(\theta_1+\theta_2)+(t-T))\right)^{-1}\hat{X}_t dt + dW_t. 
\end{equation}

\subsection{Example 2}\label{seq:LQR-example}
This example extends the previous one by adding $X_t$ into the control variable. Consider the forward fSDE
\begin{equation}
	dX_t=(aX_t+bU_t)dt+W_t, \, X_0\in\mathcal{A}_0^{sa}
\end{equation}
and $a,b,c\in\R$ and $\alpha,\beta\in\R^+$.
Define the quadratic cost function by
\begin{equation}\label{eq:LQR-costfunctional-2}
	J(U)=-\varphi\left(1/2\int_0^T(\alpha X_s^2 + \beta U_s^2)ds + \theta_1 X_T^2+\theta_2\varphi(X_T)X_T\right).
\end{equation}
We are seeking $\hat{U}\in\mathcal{U}$ such that $J(\hat{U})=\sup_{U\in\mathcal{U}}J(U)$.\\
Define the generalized Hamiltonian in $\R$ as
\begin{equation*}
	\begin{aligned}
		H(x,u,y,z)&=1/2(ax+bu)y+1/2y(ax+bu) + \Sigma(z) -  \alpha/2x^2 -  \beta/2u^2
	\end{aligned}
\end{equation*}
where
\begin{align*}
	H_x(x,u,y,z)&=ay-\alpha x\\
	H_u(x,u,y,z)&=by-\beta u.
\end{align*}
The BSDE is
\begin{equation}\label{eq:Example2-BSDE}
	-dY_t=(aY_t-\alpha X_t)dt - Z_t\#dW_t, \, Y_T=-2(\theta_1+\theta_2)X_T.
\end{equation}
Since we set $H_u=0$ we obtain $U_t=1/\beta Y_t$. As in the previous example, we make the following ansatz. Search $p\in C^1([0,T;\R])$ such that $\hat{Y}_t=p(t)\hat{X}_t$. Then
\begin{align*}
	d\hat{Y}_t&=dp\hat{X}_t+pd\hat{X}_t\\
	&=dp\hat{X}_t+ap\hat{X}_tdt + bp\hat{U}_tdt + pcdW_t\\
	&=dp\hat{X}_t+ap\hat{X}_tdt + p^2b^2/\beta \hat{X}_tdt + pcdW_t\\
	&\overset{!}{=}-ap\hat{X}dt+\alpha\hat{X}dt + \hat{Z}_t\#dW_t.
\end{align*}
We obtain the Riccati differential equation
\begin{align*}
	\dot{p}+2ap+b^2/\beta p^2-\alpha=0, \, p(T)=-2(\theta_1+\theta_2).
\end{align*}
The solution allows to calculate $\hat{U}_t=\hat{Y}_t=p\hat{X}_t$. Inserting $\hat{U}_t$ into the fSDE gives the fSDE w.r.t. $\hat{X}_t$, which is an Ornstein-Uhlenbeck process.
Similar to the commutative, classical case, $\hat{U}_t$ does not depend on the parameter $c$. 

\subsection{Example 3}
Given $b,r,\sigma,\alpha,\beta,\lambda \in\R$ consider the forward equation
\begin{equation*}
	dX_t=(rX_t+(b-r)U_t)dt+\sigma U_tdW_t + dW_t\sigma U_t,
\end{equation*}
Define
\begin{equation*}
	J(U)=\varphi\left(-1/2\int_0^T\alpha X_s^2 + \beta U_s^2ds -   (X_T-\lambda\textbf{1})^2 \right).
\end{equation*}
Setting $\alpha=\beta=0$, this example is the free equivalent of \cite[Section 6.6.2]{Pham2009}.
The Hamiltonian is
\begin{equation*}
	H(x,u,y,z)=1/2rxy+1/2ryx+1/2(b-r)uy+1/2(b-r)yu+\Sigma-\alpha/2x^2-\beta/2u^2.
\end{equation*}
Then $H_x=0$ and the BSDE is
\begin{equation*}
	-dY_t=(rY_t-\alpha X_t)dt - Z_t\#dW_t, \, Y_T=-2(X_T-\lambda\textbf{1}).
\end{equation*}
Maximizing $H$ in $\R$ by calculating $H_u=0$ yields the control variable
\begin{equation}\label{eq:examplePham_maxh_u}
	\hat{U}_t=1/\beta(b-r)\hat{Y}_t+1/\beta\Sigma_u.
\end{equation}
We consider the ansatz $\hat{Y}_t=p\hat{X}_t+q\textbf{1}$ with $p,q \in C^1([0,T],\R)$. Then 
\begin{align}
	d\hat{Y}_t&=dp\hat{X}_t+pb\hat{X}_tdt + p(b-r)U_tdt+p\sigma\hat{ U} dW + dWp\sigma \hat{U}_t\label{eq:examplePham-compcoeff-1}+dq\textbf{1}\\
	&\overset{!}{=}-r\hat{Y}_tdt-\Sigma_xdt+\alpha\hat{X}_tdt + \hat{Z}_t\#dW_t.\label{eq:examplePham-compcoeff-2}
\end{align}
This shows that
\begin{equation*}
	\hat{Z}_t=A_1\otimes B_1 + A_2\otimes B_2 
	=p\sigma\hat{U}_t\otimes\textbf{1}+\textbf{1}\otimes p\sigma\hat{U}_t
\end{equation*}
On the other hand
\begin{align*}
	\Sigma&=p\sigma/2(\sum_j\hat{U}_tB_j\varphi(A_j)+B_j\varphi(\hat{U}_tA_j)\\
	&+p\sigma/2\sum_j A_j\varphi(B_j\hat{U}_t)+A_j\hat{U}_t\varphi(B_j)\\
	&=p^2\sigma^2(2\hat{U}_t\varphi(\hat{U}_t)+\hat{U}_t^2+\varphi(\hat{U}_t^2)).
\end{align*}
Thus $\Sigma_x=0$ and 
\begin{align*}
	\Sigma_u&=4p^2\sigma^2(\hat{U}+\varphi(\hat{U})).
\end{align*}
This gives an implicit equation for $\hat{U}_t$ in \eqref{eq:examplePham_maxh_u}. To proceed we only search control $U_t$ for which $\varphi(\hat{U}_t)=0$. Then from \eqref{eq:examplePham_maxh_u} we obtain
\begin{equation*}
	\hat{U}_t=\frac{b-r}{\beta-4\sigma^2 p^2}\hat{Y}_t= \frac{p(b-r)}{\beta-4\sigma^2 p^2}\hat{X}_t+\frac{q(b-r)}{\beta-4\sigma^2 p^2}.
\end{equation*}
By comparison of coefficients in \eqref{eq:examplePham-compcoeff-1} and \eqref{eq:examplePham-compcoeff-1} we obtain
\begin{align*}
	\dot{p}+p\left(2r+\frac{p(b-r)^2}{\beta-4\sigma^2p^2}\right)-\alpha&=0, \, p_T=-2\\
	\dot{q} + q\left(r+\frac{p(b-r)^2}{\beta-4\sigma^2 p^2}\right)&=0, \, q_T=2\lambda.
\end{align*}
Note that the ODE for $p$ corresponds to \eqref{eq:matrixriccati-classical}, when $A,B,D,R$ are multiples of the identity matrix (in our case, we have two diffusion terms!).

\bibliographystyle{plainurl}
\bibliography{biblio-optim}

\end{document}